\documentclass[10pt]{amsart}
\usepackage{amsthm,amsmath,amssymb}
\usepackage[numbers]{natbib}
\usepackage{times}
\usepackage{indentfirst}
\usepackage{calrsfs}
\usepackage[usenames]{color}
\usepackage[colorlinks]{hyperref}

\newtheorem{Theorem}{Theorem}

\newtheorem{Lemma}[Theorem]{Lemma}
\newtheorem{Proposition}[Theorem]{Proposition}

\newtheorem{Example}[Theorem]{Example}
\theoremstyle{definition}
\newtheorem{definition}[Theorem]{Definition}
\def\P{{\mathbb P}}
\def\E{{\mathbb E}}
\def\R{{\mathbb R}}
\def\I{{\mathbb I}}

\begin{document}
\title{Maxima over random time intervals\\
for heavy-tailed compound renewal and L\'evy processes}
\author{Sergey Foss}
\address{Actuarial Mathematics and Statistics,
Heriot-Watt University, Edinburgh EH14 4AS, Scotland, UK}
\address{Sobolev Institute of Mathematics, Novosibirsk, Russia}
\email{S.Foss@hw.ac.uk}

\author{Dmitry Korshunov}
\address{Department of Mathematics and Statistics, Lancaster University, UK}
\email{d.korshunov@lancaster.ac.uk}

\author{Zbigniew Palmowski}
\address{Faculty of Pure and Applied Mathematics,
Wroc\l aw University of Science and Technology,
Wyb. Wyspia\'nskiego 27, 50-370 Wroc\l aw, Poland}
\email{zbigniew.palmowski@gmail.com}\maketitle

\date{\today}

\begin{abstract}
We derive subexponential tail asymptotics for the distribution of the maximum
of a compound renewal process with linear component
and of a L\'evy process, both with negative drift,
over {\it random} time horizon $\tau$
that does not depend on the future increments of the process.
Our asymptotic results are {\it uniform} over the whole class of such random times.
Particular examples are given by stopping times and by $\tau$ independent of the processes.
We link our results with random walk theory.

\end{abstract}

%\subjclass[2010]{60G51, 60G50, 60F10}

\pagestyle{myheadings} \markboth{\sc  S.\ Foss --- D.\ Korshunov --- Z.\ Palmowski}
{\sc Maxima sampling on random time intervals}

\section{Introduction and main results}\label{intro}

In this paper we derive subexponential tail asymptotics
for the distribution of the maximum of compound renewal processes
with linear component and for L\'evy processes, both with negative drift,
over random time horizon $\tau$ (which may be infinite with positive probability)
that does not depend on the future increments of $X$. We focus on obtaining
results that are {\it uniform in a broad class of random times} that is introduced below.

We believe that this is the first work in the context of the subexponential asymptotics
in the continuous-time setting where
the concepts of random times that do not depend on the future
increments of the process are introduced and systematically studied;
see Definitions \ref{indinc} and \ref{defindfuture} and Examples \ref{example1}
and \ref{example2} below.
To handle this, we propose a new approach based on the creation of some special i.i.d. cycles.
This is a crucial step
that allows us to obtain the derived results uniformly over the aforementioned random times.
We like to underline as well that our approach  produces the respective asymptotics
for random walks and for general L\'evy processes.
As such, it seems to be new for the latter class of processes, where
the Wiener-Hopf factorisation and ladder process arguments have been used before.

Let us firstly recall some basic definitions related to subexponential distributions.
A probability distribution $F$ on the real line has a {\it heavy (right) tail} if
\begin{align*}
\int_{-\infty}^{\infty} e^{\varepsilon x} F(dx) = \infty\quad\mbox{for all } \varepsilon >0,
\end{align*}
and a {\it light tail}, otherwise.
In this paper we focus on a very important sub-class of heavy-tailed distributions,
namely on the class of {\it strong subexponential} distributions ${\mathcal S}^*$.
A distribution $F$ on $\R$ with a finite mean and unbounded support
on the right belongs to the class ${\mathcal S}^*$ if
\begin{eqnarray*}
\int_0^x \overline F(x-y)\overline F(y)dy
\sim 2a^+\overline F(x)\quad\mbox{ as }x\to\infty,
\end{eqnarray*}
where $\overline F(x)=F(x,\infty)$ and $a^+=\int_0^\infty\overline F(y)dy$.
By  \cite[Theorem 3.27]{FKZ}, any $F\in {\mathcal S}^*$ is {\it subexponential},
i.e. satisfies the following two properties:
$\overline{F*F}(x)\sim 2\overline{F}(x)$ and $F$ is {\it long-tailed},
i.e. $\overline{F}(x+1) \sim \overline{F}(x)$.
Here `$*$' denotes the convolution operator and,
for any two eventually positive functions $f(x)$ and $g(x)$,
we write $f(x)\sim g(x)$ if $f(x)/g(x)\to 1$ as $x\to\infty$.
We refer to \cite{FKZ} for an overview of the properties of strong subexponential
and related distributions.

To define properly the family of (possibly improper) random times of interest,
we use the notation $\sigma(\Xi)$ for the sigma-algebra generated
by a family of random variables $\Xi$.

\begin{definition}\label{indinc}
We say that a random time $\tau\in[0,\infty]$ {\it does not depend on the future
increments of the process} $X=\{X_t, t\ge 0\}$ if the sigma-algebras
$\sigma\{X_s,\ s\le t,\ \I\{\tau\le t\}\}$
and $\sigma\{X_{t+v}-X_t,\ v>0\}$ are independent, for all $t>0$.
\end{definition}

Equivalently, the independence of the future property may be
defined as independence of the sigma-algebras
$\sigma\{X_s,\ s<t,\ \I\{\tau<t\}\}$ and $\sigma\{X_{t+v}-X_{t-0},\ v\ge 0\}$, for all $t>0$.
A similar term appeared in \cite{KP,Dima} while proving generalisations of
Wald's identity and in \cite[Section 7]{Dima2} where a random walk was considered.

For a random process $X$ {\it with independent increments},
there are two important examples of
random times that are independent of the future increments of $X$:

(i) any hitting time $\tau$ (i.e.\ the first time to hit a certain set), and,
more generally, any stopping time $\tau$ where the event $\{\tau \le t\}$
belongs to the sigma-algebra $\sigma\{X_s,\ s\le t\}$ for all $t$, and

(ii) any random time $\tau$ that does not depend on the process $X$.\\
Notice that, for a non-deterministic process $X$, a stopping time $\tau$ does not depend on
$X$ if and only if $\tau$ is a constant time.
More examples and comments related to Definition \ref{indinc}
are given in Section \ref{sec:examples}.

While Definition \ref{indinc} works well for L\'evy processes,
it is not so for compound renewal processes, in general. For example, if we follow Definition \ref{indinc}, then a constant $\tau=c$,
while being independent of any process, appears to depend on the future
increments of a compound renewal process if the process in not compound Poisson.
Indeed, the time after $c$ to the next jump (i.e. the overshot of the underlying
renewal process at time $c$) does depend on
$\sigma\{X_s,\ s\le c,\ \I\{\tau\le c\}\}=\sigma\{X_s,\ s\le c\}$.
Having this observation in mind,
we introduce another notion of `independence of the future',
that is with respect to a sequence of (random) embedded epochs.

\begin{definition}\label{defindfuture}
Given an increasing sequence of random times  $T_n$,
we say that $\tau\in[0,\infty]$ is {\it independent of the future increments of $X$
with respect to} $\{T_n\}$
%(in short, $\tau$ {\it is $\{T_n\}$-IF})
if the following condition holds: for all $n\ge 1$, the two $\sigma$-algebras
\begin{eqnarray}\label{tau.n.crp}
\qquad
\sigma\{X_s,\ s<T_n,\ \I\{\tau<T_n\}\} \
\text{and} \
\sigma\{X_{T_n+v}-X_{T_n-0},\ v\ge 0\} \
\text{are independent.}
\end{eqnarray}
\end{definition}

Below are two important examples of  stochastic processes we are mainly interested
in.

\begin{Example}\label{example1}\rm
Let $X$ be a
{\it compound renewal process with linear component},
\begin{eqnarray}\label{compound}
X_t &=& \sum_{i=1}^{N_t}Y_i+ct, \quad t\ge 0,
\end{eqnarray}
where $c$ is a real constant, $\{Y_n\}_{n\ge1}$ i.i.d.\
jump sizes  with distribution $F$ and finite mean $b$, and
$N=\{N_t, t\ge 0\}$ a renewal process independent of the jump sizes,
with jump epochs
$0=T_0<T_1<T_2<\ldots$, where $T_n-T_{n-1}>0$
are i.i.d.\ positive random variables with finite mean $1/\lambda$.\\
If a random time
$\tau$ does not depend on the future increments of $X$,
then $\tau$ satisfies the condition \eqref{tau.n.crp} with $T_n$ the jump epochs.
Indeed, for any events $A\in\sigma\{X_s,\ s<T_n,\ \tau<T_n\}$ and
$B\in\sigma\{X_s-X_{T_n-0},\ s\ge T_n\}$ conditioning on $T_n$ implies a.s.\ equality
\begin{eqnarray*}
\P\{AB\mid T_n\} &=& \P\{B\mid A,T_n\}\P\{A\mid T_n\}\ =\
\P\{B\}\P\{A\mid T_n\},
\end{eqnarray*}
where $\P\{B\mid T_n\}$ does not depend on the value of $T_n$
and $A$ due to the renewal structure of $X_t$ and hence
\begin{eqnarray*}
\P\{AB\} &=& \E\P\{AB\mid T_n\}
\ =\ \P\{B\} \E \P\{A\mid T_n\}
\ =\ \P\{B\}\P\{A\} ,
\end{eqnarray*}
so the events $A$ and $B$ are indeed independent.
\end{Example}

\begin{Example}\label{example2}\rm
Let $X$ be a L\'evy process starting at the origin
and let $\tau$ not depend on the future increments of $X$.
Fix an $\varepsilon>0$ and consider the sequence of all jump epochs $T_n$
with jump sizes $|X_{T_n}-X_{T_n-0}|\ge\varepsilon$, so $T_{n+1}-T_n$, $n\ge 0$,
are i.i.d.\ exponentially distributed random variables, hereinafter $T_0=0$.
Then $\tau$ satisfies the condition \eqref{tau.n.crp}.
Indeed, since $\tau$ does not depend on the future increments of $X$,
it does not depend on the future jumps of $X$ of size at least $\varepsilon$,
which in turn yields \eqref{tau.n.crp} by arguments similar to those used above.
\end{Example}

In this paper we firstly focus on the compound renewal process $X$
with linear component introduced in \eqref{compound}.
We assume throughout that the drift of the process is negative, that is,
\begin{eqnarray}\label{cpp.drift}
\E(cT_1+Y_1) =: -a &<& 0
\end{eqnarray}
(equivalently, $c+b\lambda=-a\lambda<0$),
which implies that the family of distributions of the partial maxima
\begin{equation}\label{Mt}
M_t := \max_{u\in[0,t]}X_u
\end{equation}
is tight and
$$
\sup_{t>0}\,\P\{M_t>x\} \le \P\{M_\infty>x\} \to 0 \quad\mbox{as }x\to\infty.
$$
We are interested in the tail behaviour of $M_\tau$ for random times $\tau$.
Our first main result here is as follows.

\begin{Theorem}\label{th:stopping.1.crp}
Let the drift condition \eqref{cpp.drift} hold.
Let the distribution $F$ of $Y_1$ belongs to ${\mathcal S}^*$.
If $c\le 0$, then
\begin{eqnarray*}
\P\{M_\tau>x\} &=&
\frac{1+o(1)}{a} \E\int_x^{x+aN_\tau}\overline F(y)dy
\quad\mbox{as }x\to\infty,
\end{eqnarray*}
uniformly for all random times $\tau\in[0,\infty]$
that satisfy \eqref{tau.n.crp} in Definition \ref{defindfuture}.
If $c>0$ and
\begin{eqnarray}\label{cond.o}
\P\{cT_1>x\}=o(\overline F(x))\quad\mbox{as }x\to\infty,
\end{eqnarray}
then
\begin{eqnarray*}
\P\{M_\tau>x\} &=&
\frac{1+o(1)}{a} \E\int_x^{x+aN_\tau}\overline F(y)dy+O(\P\{cT_1>x\}).
\end{eqnarray*}
\end{Theorem}

Hereinafter, we write $f(x,\tau)=o(g(x,\tau))$ as $x\to\infty$
uniformly for all $\tau$ if
$$
\sup_\tau\biggl|\frac{f(x,\tau)}{g(x,\tau)}\biggr|\to 0\quad\mbox{as }x\to\infty.
$$

Note that a result similar to Theorem \ref{th:stopping.1.crp}
has been obtained recently in the case of non-random times.
%(a more general than renewal setting studied by Tang in \cite{Tang2004} where the author
%only assumed that the counting process $N_t$ satisfies a law of large numbers):

\begin{Theorem}[{\cite[Theorem 1]{K2018}}]\label{th:stopping.3.crp}
Under the conditions of Theorem \ref{th:stopping.1.crp},
\begin{eqnarray*}
\P\{M_t>x\} &\sim&
\frac{1}{a} \int_0^{a\E N_t}\overline F(x+y)dy
\quad\mbox{as }x\to\infty\mbox{ uniformly for all }t\in(0,\infty].
\end{eqnarray*}
\end{Theorem}

One can see the difference in results of Theorems \ref{th:stopping.1.crp}
and \ref{th:stopping.3.crp}:
in Theorem \ref{th:stopping.1.crp}, the expectation of the integral is taken
while in Theorem \ref{th:stopping.3.crp},
for constant $\tau$, it is simplified to an integral with constant upper limit.
We have to comment that, despite the fact that the results look similar,
the extension to random times requires a lot of effort in the case where the
random time $\tau$ depends on the process -- say like stopping times.
Our approach is based on the introduction of special i.i.d.\ cycles
that allows us to handle random times that do not depend
on the future increments.
We like to underline as well that our approach  produces the respective asymptotics
for random walks and for general L\'evy processes.
Note also that such an extension is really needed in many applications,
say in risk or in queueing theory where we often need to know how likely
it is to exceed a high level within some time cycle or similar random time interval.

A more general than renewal setting was studied by Tang \cite{Tang2004}
where the author obtained uniform asymptotics for the supremum over time intervals $[0,t]$
for a counting process that satisfies a law of large numbers.

Let us consider a {\it random walk}  $S_0=0$, $S_n:=Y_1+\ldots+Y_n$, $n\ge 1$
and its partial maxima $M_n:=\max(S_0,\ldots,S_n)$.
Since a random walk may be considered as a particular case of a compound
renewal process with constant jump epochs $T_n=n$,
we derive from Theorem \ref{th:stopping.1.crp} the following result
which replicates and corrects several results from
\cite{Asmussen, FTZ, FPZ, FZ, Greenwood, GreenwoodMonroe, KKM};
see also \cite{FKZ} for further references.

\begin{Theorem}\label{th:stopping.1}
Let $\E Y_1=:-a<0$. If the distribution $F$ of $Y_1$ belongs to ${\mathcal S}^*$, then
\begin{eqnarray*}
\P\{M_\tau>x\} &=&
\frac{1+o(1)}{a} \E\int_0^{a\tau}\overline F(x+y)dy\quad\mbox{as }x\to\infty,
\end{eqnarray*}
uniformly for all counting random variables $\tau\in[0,\infty]$
that do not depend on the future jumps of $S_n$.
\end{Theorem}

First results on the asymptotics of randomly stopped sequences with independent increments
are due to Greenwood \cite{Greenwood} and
Greenwood and Monroe \cite{GreenwoodMonroe}
where a case of a bounded or regularly varying at infinity stopping time $\tau$
and regularly varying at infinity $F$ is considered.

Asmussen \cite{Asmussen} obtained this result for the
proper hitting time $\tau=\min\{n\ge 1: S_n\le 0\}$
with $\E\tau=1/\P\{S_\infty\le 0\}$, where the right hand side
is asymptotically equal to $\E\tau\overline F(x)$.

In \cite{FPZ}, a more general statement than Theorem \ref{th:stopping.1} is presented,
it concerns up-crossing by a random walk of a family of high level non-linear boundaries,
and its proving technique is based on Asmussen's \cite{Asmussen} result on the tail asymptotics
for the maximum of a random walk until it hits the negative half-line.
However, the main weakness of the whole paper \cite{FPZ} is that it deals
with both stopping times (that are deterministic functions of the trajectory
of the random walk) and times independent of the future
(in the sense of Definition 1 above), with naming both random times
as ``stopping times'' and omitting their formal definitions.
This leads to confusion and uncertainties to the reader,
and to incompleteness of the  proof of the key Lemma 1 there.
%However, its proof is based on a key Lemma 1 whose proof is incomplete and is presented not in the best way.
%Namely,
The lemma states some asymptotic results which hold uniformly
for all stopping times from a family $\mathcal T_\varphi:=\{\sigma:0\le\sigma\le\varphi\}$
where $\varphi\ge 0$ is a stopping time such that $\E\varphi<\infty$,
i.e.\ the family of stopping times possesses an integrable majorant.
Then its proof starts with saying that, without loss of generality,
it suffices to assume that $\sigma\ge 1$, and then this assumption is essentially
used, together with the existence of an integrable majorant.
However, the latter reduction is not supported by an argument and may fail
if the top random variable is an independent of the future time which is not a stopping time.
In particular, the minimum of two independent of the future times
may depend on the future, and this should be avoided
(see more comments and remarks in Section 2).
This whole issue can be fixed by means of Lemma \ref{l:conditioning} below.

Our proof of Theorem \ref{th:stopping.1.crp} contains also novel methodological arguments,
in the sense that it does not use the down- and upcrossing arguments as
in \cite{Asmussen} and then in \cite{FPZ} and is based on the local renewal theorem instead,
that makes the proof self-contained and more straightforward.

Notice that Theorem \ref{th:stopping.1} generalises the following result for partial maxima
of a random walk over non-random time intervals.

\begin{Theorem}[{\cite[Theorem 5.3]{FKZ}}]\label{th:stopping.3}
Under the conditions of Theorem \ref{th:stopping.1},
\begin{eqnarray*}
\P\{M_n>x\} &\sim&
\frac{1}{a} \int_0^{an}\overline F(x+y)dy
\quad\mbox{as }x\to\infty\mbox{ uniformly for all }n\ge 1.
\end{eqnarray*}
\end{Theorem}

Now let $X$ be a {\it L\'evy process} starting at the origin, that is,
a c\`adl\`ag stochastic process
(i.e. almost all its paths are right continuous with existing left limits everywhere)
with stationary independent increments,
where stationarity means that, for $s<t$, the probability distribution of
$X_t-X_s$ depends only on $t-s$ and where the independence of increments means that,
for all $k\ge2$ and for all $s_1<t_1\le s_2<t_2\le\ldots\le s_k <t_k$,
the differences $X_{t_i}-X_{s_i}$, $1\le i \le k$ are mutually independent random variables.

Stemming from Theorem \ref{th:stopping.1.crp} above,
by considering the compound Poisson component of $X$ with linear component,
one can find the correct asymptotic lower and upper bounds for the tail of $M_\tau$.
%However this approach does not work so well for the matching upper bound,
%because the suprema of the Gaussian component and the compound
%Poisson component are dependent via $\tau$. For that reason we need to follow
%an alternative approach to derive the matching upper tail bound.
Our main result for a L\'evy process is the following theorem.

\begin{Theorem}\label{th:stopping.1.lp}
Let $\E X_1=:-m<0$. If the distribution of $X_1$ belongs to ${\mathcal S}^*$,
then, for some $\beta>0$,
\begin{eqnarray*}
\P\{M_\tau>x\} &=&
\frac{1+o(1)}{m} \E\int_x^{x+m\tau}\P\{X_1>y\}dy+O(e^{-\beta x})
\quad\mbox{as }x\to\infty,
\end{eqnarray*}
uniformly for all random times $\tau\in[0,\infty]$
that do not depend on the future increments of $X$.
\end{Theorem}

First results on the asymptotics of randomly stopped L\'evy processes
go back to Greenwood and Monroe \cite{GreenwoodMonroe} who considered
a particular case of regularly varying at infinity distributions $F$ and stopping times $\tau$.
The case of independent sampling was considered in Korshunov \cite[Theorem 10]{K2018}.

It has been suggested by Asmussen and Kl\"uppelberg \cite[Sect. 1.1]{AK}
and by Asmussen \cite[Sect. 2.4]{Asmussen} to follow a discrete skeleton argument
in order to prove these asymptotics for $\tau=\infty$
when the integrated tail of the L\'evy measure is subexponential;
notice that this approach requires additional considerations
which take into account fluctuations of L\'evy processes within time slots.

In Doney et al.\ \cite{DKM} the passage time problem has been considered
for L\'evy processes, emphasising heavy tailed cases;
local and functional versions of limit distributions are derived
for the passage time itself, as well as for the position
of the process just prior to passage, and the overshoot of a high level
which is an extension for L\'evy processes of corresponding results
for random walks, see e.g.\ Foss et al.\ \cite[Theorem 5.24]{FKZ}.

Notice that the conclusions of both Theorems \ref{th:stopping.1.crp} and \ref{th:stopping.1.lp}
are not so trivial. For example, one could doubt if they really hold true uniformly for all $\tau$.
At first glance, the uniformity may fail for the family of hitting times
$\tau(x)=\inf\{t\ge 0: X_t>x\}$, $x>0$. However, for such $\tau(x)$, on the one hand
\begin{eqnarray*}
\P\{M_{\tau(x)}>x\} &=& \P\{\tau(x)<\infty\}
\ =\  \P\{M_\infty>x\}\ \sim\ \frac{1}{m}\int_x^\infty \overline F(y)dy
\quad\mbox{as }x\to\infty,
\end{eqnarray*}
while on the other hand,
\begin{eqnarray*}
\E\int_x^{x+m\tau(x)}\overline F(y)dy &\le& \int_x^\infty \overline F(y)dy,
\end{eqnarray*}
and
\begin{eqnarray*}
\E\int_x^{x+m\tau(x)}\overline F(y)dy &\ge& \P\{\tau(x)=\infty\}\int_x^\infty \overline F(y)dy
\ \sim\ \int_x^\infty \overline F(y)dy,
\end{eqnarray*}
due to $\P\{\tau(x)=\infty\}\to 1$ as $x\to\infty$, so
\begin{eqnarray*}
\E\int_x^{x+m\tau(x)}\overline F(y)dy &\sim& \int_x^\infty \overline F(y)dy
\quad\mbox{as }x\to\infty.
\end{eqnarray*}

The paper is organised as follows.
In Section \ref{sec:examples} the property of the independence of the future increments
given in Definition \ref{indinc} is discussed in detail.
In Section \ref{taudep.crp} we prove Theorem \ref{th:stopping.1.crp}
on the compound renewal process.
For a compound Poisson process, the tail asymptotics may be significantly
improved in what concerns the upper limit
of the integral, it is done in Section \ref{taudep.cPp}.
Next, in Section \ref{taudep.lp}
we prove Theorem \ref{th:stopping.1.lp} on L\'evy processes.
Finally, Appendix includes the proofs of three auxiliary results.

\section{Independence of the future increments -- discussion and  further examples}
\label{sec:examples}

We discuss now Definitions \ref{indinc} and \ref{defindfuture}.
Firstly observe the following properties of independent of the future increments random times:
\begin{itemize}
\item   while the minimum of two stopping times is a stopping time,
the same property for independent of the future increments random times fails, in general;
\item   the minimum of an independent of the future increments random time
and of a stopping time does not depend on the future increments, too;
\item    for a general compound renewal processes (with a possible linear drift),
an independent time (hence e.g.\ a constant time)
depends on the future increments, in the sense of (more-or-less standard) Definition \ref{indinc}.
We have eliminated this conceptual problem
by introducing a novel Definition \ref{defindfuture}, that involves embedded time instants.
Note that, in particular cases, one can deal with constant times using
another approach based on lower and upper bounds, see e.g. Example 6 in \cite{FPZ}.
\end{itemize}

We give a number of examples that clarify our results and the two definitions
of the independence of the future.
We focus first on the case when $X$ is a compound renewal process
with linear component defined in \eqref{compound}
and when all the assumptions of Theorem \ref{th:stopping.1.crp} hold true.

\begin{Example}\label{example3.1.1}\rm
If $\tau$ does not depend on the future increments of $X$
in the sense of Definition \ref{defindfuture}, then the two $\sigma$-algebras
$\sigma\{X_s,\ s<T_n,\ \I\{\tau=T_n\}\}$ and $\sigma\{X_{T_n+v}-X_{T_n-0},\ v\ge 0\}$
are independent for all $n\ge 1$.
\end{Example}

\begin{Example}\label{example3.1}\rm
Any random time $\tau<T_1$ a.s. does not depend on the future increments of $X$
in the sense of Definition \ref{defindfuture}. In this case $N_\tau=0$
and the integral in the conclusion of Theorem \ref{th:stopping.1.crp} is zero.
This example shows why the additional term $O(\P\{cT_1>x\})$
cannot be avoided in the case $c>0$ in Theorem \ref{th:stopping.1.crp}, in general.
\end{Example}

\begin{Example}\label{example3}\rm
Let $\{\tau_n\}_{n\ge 1}$ be a family of non-negative random variables independent of $X$
with $p_n:=\P\{\tau_n =\infty\}$. % and $E_n:=\E\{\tau_n \mid \tau_n <\infty\}$.
Then,
\begin{eqnarray*}
\P\{M_{\tau_n}>x\} &=&
\frac{p_n+o(1)}{a}\int_x^\infty \overline F(y)dy
+\frac{1-p_n}{a}\E \Bigl\{\int_x^{x+aN_{\tau_n}}\overline F(y)dy\Big |\tau_n<\infty\Bigr\}
+o(\overline F(x))\quad\mbox{as }x\to\infty,
\end{eqnarray*}
uniformly for all $n\ge 1$. If additionally $\sup_n \E\{\tau_n\mid\tau_n<\infty\}<\infty$
and $\inf_np_n>0$, then,
since $\overline F(x)=o(\int_x^\infty \overline F(y)dy)$, we have
\begin{eqnarray*}
\P\{M_{\tau_n}>x\} &\sim&
\frac{p_n}{a}\int_x^{\infty}\overline F(y)dy
\quad\mbox{as }x\to\infty\ \mbox{uniformly for all }n\ge 1.
\end{eqnarray*}
\end{Example}

\begin{Example}\label{example8}\rm
%Let $X$ be a compound renewal process given in Example \ref{example1} above.
Assume that $\{T_n-T_{n-1}\}$ are non-degenerate;
let $a>0$ be such that $\P\{T_n-T_{n-1}>a\} \in (0,1)$.
Let $\nu_1 < \nu_2 < \ldots$\ be the consecutive
times when $T_{\nu_k}-T_{\nu_k-1} >a$.  Let
$\eta$ be an independent non-degenerate counting random variable.
Then $\tau = T_{\nu_{\eta}}$ satisfies Definition \ref{defindfuture}.
\end{Example}

Now we present a number of examples of independent of the
future random times that are not stopping times.
We start with examples for random walks.
Let $S_n = \sum_{i=1}^n \xi_i$, $n\ge 0$ be a random walk
with i.i.d.\ increments.

\begin{Example}\label{example6}\rm
Let $\P\{\xi_1>0\}>0$ and $\P\{\xi_1<0\}>0$.
Let $\eta$ be an independent non-degenerate counting random variable, let
$\tau_1 := \min \{n>0: \xi_n <0 \}$ and let
$\tau_k := \min \{ n>\tau_{k-1}: \xi_n<0\}$ for $k\ge 2$.
Then a proper random variable $\tau =\tau_{\eta}$ does not depend
on the future of random variables $\{\xi_n\}$, but it is not a stopping time.
\end{Example}

\begin{Example}\label{example.hitting.r}\rm
Let $B\in\mathbb B(\R)$. Define hitting times $\tau_1:=\min\{n\ge 1: S_n\in B\}$
and, for all $k\ge 2$, $\tau_k:=\min\{n>\tau_{k-1}: S_n\in B\}$.
Let $\eta$ be an independent counting random variable.
Then (possibly improper) random variable $\tau_\eta$ does not depend
on the future of random variables $\{\xi_n\}$, but it is not a stopping time.
\end{Example}

\begin{Example}\label{example7}\rm
Let $g$ be a measurable function on the real line,
taking values in $(0,1)$ and  such that $g(x)\neq g(y)$ for $x\neq y$.
Let  $\{\zeta_z\}_{z\in (0,1)}$ be a family of geometric random variables,
with $\zeta_z$ having parameter $z$, that do not depend on the random walk.
Then $\tau = 1 +\zeta_{g(\xi_1)}$
is a random time that does not depend on the future of random variables $\{\xi_n\}$,
but is not a stopping time.
\end{Example}

\begin{Example}\label{example9}\rm
Let $X_t$ be a L\'evy process. For a fixed $a>0$, let
$\tau_1 := \inf \{t>0: \Delta X_t >a \}$ and let
$\tau_k := \inf\{t>\tau_{k-1}: \Delta X_t >a\}$ for $k\ge 2$.
Then, for an independent non-degenerate counting random variable $\eta$,
$\tau_\eta$ satsfies Definition \ref{indinc}.
\end{Example}

\section{Proof of Theorem \ref{th:stopping.1.crp} for compound renewal process with linear drift}
\label{taudep.crp}

The proof of Theorem \ref{th:stopping.1.crp} is split into two parts,
the lower bound is considered in Proposition \ref{l:stopping.+2.crp}
and the upper bound in Proposition \ref{l:stopping.+3.crp}.

\begin{Proposition}\label{l:stopping.+2.crp}
Let $-a=\E(cT_1+Y_1)<0$. If $F$ is long-tailed, then
\begin{eqnarray*}
\P\{M_\tau>x\} &\ge&
\frac{1+o(1)}{a} \E\int_x^{x+aN_\tau}\overline F(y)dy
\quad\mbox{as }x\to\infty,
\end{eqnarray*}
uniformly for all random times $\tau\in[0,\infty]$ that satisfy \eqref{tau.n.crp}.
\end{Proposition}

\begin{proof}
The random time $\widetilde\tau:=\max\{T_n:\ T_n\le\tau\}$
does not depend on the future increments of $X$ as $\tau$ does not.
In addition, $M_\tau\ge M_{\widetilde\tau}$
whatever the sign of $c$ and $N_\tau=N_{\widetilde\tau}$, so
\begin{eqnarray*}
\E\int_x^{x+aN_\tau}\overline F(y)dy &=&
\E\int_x^{x+aN_{\widetilde\tau}}\overline F(y)dy.
\end{eqnarray*}
Therefore,
without loss of generality it is enough to consider the case where
$\tau$ only takes values from $\{T_n,\ n\ge 0\}$.
Then, by Lemma \ref{l:conditioning}, it suffices to consider the case where $\tau\ge T_1$.

As the distribution $F$ is long-tailed, there exists a function $h(x)\uparrow\infty$
as $x\to\infty$ such that (see, e.g. \cite[Lemma 2.19]{FKZ})
\begin{eqnarray}\label{F.h.x}
\overline F(x+h(x)) &\sim& \overline F(x)\quad\mbox{as }x\to\infty.
\end{eqnarray}
Since the events $\{M_{T_n-0}\le x,X_{T_n}>x,\tau\ge T_n\}$ are disjoint
and each of them implies that $M_\tau>x$, we have
\begin{eqnarray}\label{Mtaurep.crp}
\P\{M_\tau>x\}
&\ge& \sum_{n=1}^\infty \P\{M_{T_n-0}\le x,X_{T_n}>x,\tau\ge T_n\}\nonumber\\
&\ge& \sum_{n=1}^\infty \P\{M_{T_n-0}\le x,X_{T_n-0}>-(a+\varepsilon)(n-1)-h(x),
X_{T_n}>x,\tau\ge T_n\},\nonumber\\
&\ge& \sum_{n=1}^\infty \P\{M_{T_n-0}\le x,
X_{T_n-0}>-(a{+}\varepsilon)(n{-}1){-}h(x),\nonumber\\&&\qquad Y_n>x{+}(a{+}\varepsilon)(n{-}1){+}h(x),\tau\ge T_n\},
\end{eqnarray}
for any fixed $\varepsilon>0$. Since $\{\tau\ge T_n\}=\overline{\{\tau<T_n\}}$ and
owing to the condition \eqref{tau.n.crp}, the last series equals to
\begin{eqnarray*}
\lefteqn{\sum_{n=1}^\infty \P\{M_{T_n-0}\le x,X_{T_n-0}>-(a{+}\varepsilon)(n{-}1){-}h(x),
\tau\ge T_n\}
\overline F(x+(a{+}\varepsilon)(n{-}1){+}h(x))}\\
&=& \sum_{n=1}^\infty \P\{X_{T_n-0}>-(a+\varepsilon)(n-1)-h(x),\tau\ge T_n\}
\overline F(x+(a+\varepsilon)(n-1)+h(x))\\
&&-\sum_{n=1}^\infty \P\{M_{T_n-0}>x,X_{T_n-0}>-(a+\varepsilon)(n-1)-h(x),\tau\ge T_n\}\nonumber\\&&\qquad
\overline F(x+(a+\varepsilon)(n-1)+h(x))\\
&\ge& \sum_{n=1}^\infty \P\{X_{T_n-0}>-(a+\varepsilon)(n-1)-h(x),\tau\ge T_n\}
\overline F(x+(a+\varepsilon)(n-1)+h(x))\nonumber\\&&\qquad
-\P\{M_\tau>x\}\sum_{n=0}^\infty \overline F(x{+}an).
\end{eqnarray*}
Therefore, since $\sum_{n=0}^\infty \overline F(x+an)\le a^{-1}\overline F_I(x-a)$,
it follows from \eqref{Mtaurep.crp}  that
\begin{eqnarray}\label{Mtaubelowsigma2.crp}
\P\{M_\tau>x\}
&\ge& \frac{1}{1+a^{-1}\overline F_I(x-a)}
\sum_{n=1}^\infty \P\{X_{T_n-0}>-(a+\varepsilon)(n-1)-h(x),\tau\ge T_n\}
\nonumber\\&&\qquad \overline F(x{+}(a{+}\varepsilon)(n{-}1){+}h(x))\nonumber\\
&=& (1+o(1)) \sum_{n=1}^\infty \P\{X_{T_n-0}>-(a+\varepsilon)(n-1)-h(x),\ \tau\ge T_n\}
\nonumber\\&&\qquad\overline F(x+(a+\varepsilon)(n-1)+h(x))
\end{eqnarray}
as $x\to\infty$ because $\overline F_I(x-a)\to 0$; here $o(1)$ does not depend on $\tau$.
Let us decompose the last sum as follows:
\begin{eqnarray}
\lefteqn{\sum_{n=1}^\infty \P\{\tau\ge T_n\}
\overline F(x+(a+\varepsilon)(n-1)+h(x))}\nonumber\\
&&-\sum_{n=1}^\infty \P\{X_{T_n-0}\le -(a{+}\varepsilon)(n{-}1){-}h(x),\tau\ge T_n\}
\overline F(x{+}(a{+}\varepsilon)(n{-}1){+}h(x))\nonumber\\
&=:& \Sigma_1-\Sigma_2. \label{sigma12.crp}
\end{eqnarray}
To bound the value of $\Sigma_2$, we introduce recursively a sequence
$\{\theta_k\}_{k\ge 0}$ of stopping times by letting $\theta_0=0$ and, for $k\ge 0$,
\begin{align*}
j_{k+1}
 := \min \{i : T_i>\theta_k: \ X_{T_i-0}-X_{\theta_k-0} > -(i-j_k)(a+\varepsilon)\},
\end{align*}
and $\theta_{k+1}=T_{j_{k+1}}$, so that $X_{\theta_k-0} > -j_k(a+\varepsilon)$.
By construction, $\theta_k-\theta_{k-1}$, $k\ge 1$, are i.i.d.\ random variables and
\begin{eqnarray}\label{theta.crp}
\E\theta_1<\infty.
\end{eqnarray}
Introduce the minima over disjoint time intervals
\begin{align*}
L_k := \min_{ \theta_k<T_i\le\theta_{k+1}}
(X_{T_i-0}-X_{\theta_k-0}+(i-j_k)(a+\varepsilon)),\quad k\ge 0,
\end{align*}
and notice that
\begin{eqnarray}\label{L.crp}
\{L_k,\ k\ge 0\}\quad\mbox{are i.i.d.\ proper random variables}.
\end{eqnarray}
Then, for all $n\ge 1$,
\begin{eqnarray*}
\lefteqn{\P\{X_{T_n-0}\le -(a+\varepsilon)(n-1)-h(x),\tau\ge T_n\}}\\
&=& \sum_{k=0}^{n-1} \P \{\theta_k<T_n\le \theta_{k+1},
X_{T_n-0} \le -(a+\varepsilon)(n-1)-h(x), \tau\ge T_n\}\\
&\le& \sum_{k=0}^{n-1}\P\{\theta_k<T_n\le \theta_{k+1},
X_{T_n-0}-X_{\theta_k-0}+(a+\varepsilon)(n{-}1{-}j_k) \le -h(x),\tau\ge T_n\}\\
&\le& \sum_{k=0}^{n-1} \P\{\theta_k<T_n\le \theta_{k+1},
L_k \le -h(x),\tau\ge\theta_k\}.
\end{eqnarray*}
Hence,
\begin{eqnarray*}
\Sigma_2 &\le&
\sum_{k\ge 0} \sum_{n\ge k+1}\P\{\theta_k<T_n\le \theta_{k+1}, L_k \le -h(x),\tau\ge \theta_k\}
\overline F(x{+}(a{+}\varepsilon)(n{-}1){+}h(x)) \\
&=&
\sum_{k\ge 0} \biggl(\E {\mathbb I}\{\tau\ge\theta_k\}\sum_{n\ge k+1}{\mathbb I}\{
\theta_k<T_n\le \theta_{k+1}, L_k \le -h(x)\}\biggr)
\overline F(x{+}(a{+}\varepsilon)k{+}h(x)) \\
&=& \sum_{k\ge 0} \E \left[{\mathbb I}\{\tau\ge\theta_k\}(\theta_{k+1}-\theta_k)
{\mathbb I}\{L_k \le -h(x)\}\right] \overline F(x+(a+\varepsilon)k+h(x)).
\end{eqnarray*}
By the condition \eqref{tau.n.crp},
the event $\{\tau\ge\theta_k\}=\overline{\{\tau<\theta_k\}}$
does not depend on the random variable
$(\theta_{k+1}-\theta_k){\mathbb I}\{L_k\le -h(x)\}\in\sigma\{X_t-X_{\theta_k-0}, t>\theta_k\}$,
hence we conclude that the last sum is equal to
\begin{eqnarray*}
\lefteqn{\sum_{k\ge 0} \P\{\tau\ge\theta_k\}\E\{\theta_{k+1}-\theta_k;\ L_k \le -h(x)\}
\overline F(x+(a+\varepsilon)k+h(x))}\\
&\le& \sum_{k\ge 0} \P\{\tau\ge T_k\}\E\{\theta_{k+1}-\theta_k;\ L_k \le -h(x)\}
\overline F(x+(a+\varepsilon)k+h(x)),
\end{eqnarray*}
due to $T_k\le \theta_k$. It follows from \eqref{theta.crp} and \eqref{L.crp} that
\begin{align*}
\sup_k \E\{\theta_{k+1}-\theta_k;\ L_k \le -h(x)\} &\to 0\quad\mbox{as }x\to\infty,
\end{align*}
which implies that, as $x\to\infty$,
\begin{eqnarray*}
\Sigma_2 &=& o(1) \sum_{n\ge 0} \P\{\tau\ge T_n\} \overline F(x+(a+\varepsilon)n+h(x))
\ =\ o(\overline F(x))\nonumber\\&&\qquad+o(1)
\sum_{n\ge 1} \P\{\tau\ge T_n\} \overline F(x+(a+\varepsilon)n+h(x)).
\end{eqnarray*}
Thus we derive from \eqref{Mtaubelowsigma2.crp} and \eqref{sigma12.crp} that
\begin{eqnarray*}
\P\{M_\tau>x\}
&\ge& (1+o(1)) \sum_{n=1}^\infty \P\{\tau\ge T_n\}
\overline F(x+(a+\varepsilon)(n-1)+h(x))
+o(\overline F(x))\\
&=& (1+o(1))\sum_{n=1}^\infty \P\{\tau\ge T_n\} \overline F(x+(a+\varepsilon)(n-1))
+o(\overline F(x)),
\end{eqnarray*}
as $x\to\infty$, owing to \eqref{F.h.x}; here $o(1)$ does not depend on $\tau$.
Since $\P\{\tau\ge T_1\}=1$, the last term can be incorporated into the sum,
\begin{eqnarray*}
\P\{M_\tau>x\}
&\ge& (1+o(1))\sum_{n=1}^\infty \P\{\tau\ge T_n\} \overline F(x+(a+\varepsilon)(n-1)).
\end{eqnarray*}
The last sum equals
\begin{eqnarray}\label{l.b.2.crp}
\lefteqn{\sum_{n=1}^\infty \P\{\tau\in[T_n,T_{n+1})\}
\sum_{k=1}^n \overline F(x+(a+\varepsilon)(k-1))
+\P\{\tau=\infty\}\sum_{n=0}^\infty \overline F(x+(a+\varepsilon)n)}\nonumber\\
&&\hspace{3mm}\ge\  \frac{1}{a{+}\varepsilon}\biggl(
\sum_{n=1}^\infty \P\{\tau\in[T_n,T_{n+1})\} \int_x^{x+(a+\varepsilon)n} \overline F(y)dy
+\P\{\tau{=}\infty\}\int_x^\infty \overline F(y)dy\biggr)\nonumber\\
&&\hspace{3mm}=\ \frac{1}{a+\varepsilon}\E \int_x^{x+(a+\varepsilon)N_\tau}
\overline F(y)dy
\ \ge\ \frac{1}{a+\varepsilon}
\E \int_x^{x+aN_\tau} \overline F(y)dy
\end{eqnarray}
for all random variables $\tau$ that satisfy \eqref{tau.n.crp}.
Since the  choice of $\varepsilon>0$ is arbitrary, we conclude the lower bound.
\end{proof}

\begin{Proposition}\label{l:stopping.+3.crp}
Let $\E(cT_1+Y_1)=-a\le0$ and let $F\in{\mathcal S}^*$.
If $c\le 0$ then
\begin{eqnarray*}
\P\{M_\tau>x\} &\le&
\frac{1+o(1)}{a} \E\int_0^{aN_\tau}\overline F(x+y)dy\quad\mbox{as }x\to\infty,
\end{eqnarray*}
uniformly for all random times $\tau\in[0,\infty]$ that satisfy \eqref{tau.n.crp}.
If $c>0$ and the condition \eqref{cond.o} holds, then
\begin{eqnarray*}
\P\{M_\tau>x\} &\le&
\frac{1+o(1)}{a} \E\int_0^{aN_\tau}\overline F(x+y)dy+O(\P\{cT_1>x\}).
\end{eqnarray*}
\end{Proposition}

\begin{proof}
The maximum is always attained at a jump associated epoch,
either prior to or at the jump epoch.

Firstly consider the case where $c\le 0$ and hence the maximum is attained at a jump epoch.
As in the last proof, the random time $\widetilde\tau:=\max\{T_n:\ T_n\le\tau\}$
does not depend on the future increments of $X$.
Due to $c\le 0$, $M_\tau=M_{\widetilde\tau}$. In addition, $N_\tau=N_{\widetilde\tau}$.
Therefore, without loss of generality it is enough to consider the case where
$\tau$ only takes values from $\{T_n,\ n\ge 0\}$.
Then, by Lemma \ref{l:conditioning}, it suffices to consider the case where $\tau\ge T_1$.

Define $\theta_0=0$. Fix an $\varepsilon\in(0,a)$. By the strong law of large numbers,
there exists a level $A=A(\varepsilon)<\infty$ such that the stopping time
\begin{eqnarray*}
\theta_1\ =\ \theta_1(A) &:=& \inf\{T_n:\ X_{T_n}>n(-a+\varepsilon)+A\}
\end{eqnarray*}
is finite with small probability,
\begin{eqnarray}\label{p.A.to.0}
p\ =\ p(A)\ :=\ \P\{\theta_1(A)<\infty\} &\le& \varepsilon.
\end{eqnarray}
As $X_t\le A$ for all $t<\theta_1$, by the total probability law, for $x>A$,
\begin{eqnarray*}
\P\{X_{\theta_1\wedge\tau}>x\} &=& \sum_{n=1}^\infty
\P\{\theta_1=T_n\le\tau,\ X_{T_{n-1}}\le (n-1)(-a+\varepsilon)+A,\ X_{T_n}>x\}\\
&\le& \sum_{n=1}^\infty \P\{\tau\ge T_n,\ X_{T_{n-1}}\le (n-1)(-a+\varepsilon)+A,\ X_{T_n}>x\}\\
&\le& \sum_{n=1}^\infty\P\{\tau\ge T_n,\ Y_n>x+(n-1)(a-\varepsilon)-A\}.
\end{eqnarray*}
Since $\{\tau\ge T_n\}=\overline{\{\tau< T_n\}}$ and $\tau$ satisfies \eqref{tau.n.crp},
\begin{eqnarray*}
\P\{X_{\theta_1\wedge\tau}>x\} &\le&
\sum_{n=1}^\infty\P\{\tau\ge T_n\} \overline F(x+n(a-\varepsilon)-A-a)
\ =\ \E \sum_{n=1}^{N_\tau} \overline F(x+n(a-\varepsilon)-A-a).
\end{eqnarray*}
Since $\overline F$ is a decreasing function,
\begin{eqnarray*}
\P\{X_{\theta_1\wedge\tau}>x\} &\le&
\frac{1}{a-\varepsilon} \E \int_0^{aN_\tau} \overline F(x-A-a+y)dy.
\end{eqnarray*}
Taking into account that $M_{\theta_1\wedge\tau}\le \max(A,X_{\theta_1\wedge\tau})$,
we conclude an upper bound
\begin{eqnarray}\label{eq:upper.1.crp}
\P\{M_{\theta_1\wedge\tau}>x\} &=& \P\{X_{\theta_1\wedge\tau}>x\}
\ \le\ \overline G_\tau(x),\quad x>A,
\end{eqnarray}
where the distribution $G_\tau$ on $[A,\infty)$ defined by its tail as
\begin{eqnarray*}
\overline G_\tau(x) &:=& \min\biggl(1,\frac{1}{a-\varepsilon}
\E \int_0^{aN_\tau} \overline F(x-A-a+y)dy\biggr),
\quad x\ge A,
\end{eqnarray*}
is long-tailed because $F$ is so. Hence
\begin{eqnarray}\label{eq:upper.4.crp}
\overline G_\tau(x) &\sim& \frac{1}{a-\varepsilon} \E \int_0^{aN_\tau} \overline F(x+y)dy
\quad \mbox{as }x\to\infty\mbox{ uniformly for all }\tau\ge 0.
\end{eqnarray}
Notice that, by Fubini's Theorem,
\begin{eqnarray}\label{eq:G.tau.mu}
\frac{1}{a-\varepsilon} \E\int_0^{aN_\tau}\overline F(x+y)dy &=&
\frac{1}{a-\varepsilon} \int_0^\infty \P\{N_\tau\in dz\}\int_0^{az}\overline F(x+y)dy\nonumber\\
&=& \frac{1}{a-\varepsilon} \int_0^\infty \overline F(x+y)dy \P\{N_\tau>y/a\}
\ =\ \int_0^\infty\overline F(x+y)\mu_\tau(dy),
\end{eqnarray}
where the measure
$\mu_\tau(dy) := \frac{1}{a-\varepsilon} \P\{N_\tau>y/a\}dy$
satisfies the condition $\mu_\tau(x,x+1]\le 1/(a-\varepsilon)$ for all $x$ and $\tau$.

For $k\ge 2$, define recursively stopping times $\theta_k=\theta_k(A)$ as follows:
on the event $\{\theta_{k-1}<\infty\}$,
\begin{eqnarray*}
j_k &:=& \inf\{n: T_n>\theta_{k-1}:\
X_{T_n}-X_{\theta_{k-1}}>(n-j_{k-1})(-a+\varepsilon)+A\}
\end{eqnarray*}
and $\theta_k:=T_{j_k}$. Then, by the renewal structure of the process,
$\P\{\theta_k(A)<\infty\mid\theta_{k-1}(A)<\infty\}=\P\{\theta_1(A)<\infty\} = p(A)$.

Let $k\ge 1$. For all $m\le k$,
\begin{eqnarray*}
\P\{X_{\theta_m}{-}X_{\theta_{m-1}}>x,\ \theta_k{<}\infty,\ \theta_k\le\tau\}
&\le& \P\{X_{\theta_m}{-}X_{\theta_{m-1}}>x,\ \theta_k{<}\infty,\ \theta_m\le\tau\}\\
&=&
p^{k-m}(A)\P\{X_{\theta_m}{-}X_{\theta_{m-1}}>x,\ \theta_m<\infty,\ \theta_m\le\tau\},
\end{eqnarray*}
because the event $\{X_{\theta_m}-X_{\theta_{m-1}}>x,\ \theta_m<\infty,\ \theta_m\le\tau\}
=\{X_{\theta_m}-X_{\theta_{m-1}}>x,\ \theta_m<\infty,\ \overline{\theta_m>\tau}\}$
does not depend on the future increments, see \eqref{tau.n.crp}.
By the total probability law, for $x>0$,
\begin{eqnarray*}
\lefteqn{\P\{X_{\theta_m}-X_{\theta_{m-1}}>x,\ \theta_m<\infty,\ \theta_m\le\tau\}}\\
&=& \sum_{n=1}^\infty
\P\bigl\{\theta_m\le\tau,\ \theta_m=T_{j_{m-1}+n},\
X_{\theta_m}-X_{\theta_{m-1}}>x\bigr\}\\
&\le& \sum_{n=1}^\infty
\P\bigl\{\theta_m=T_{j_{m-1}+n},\ T_{j_{m-1}+n}\le\tau,\
X_{T_{j_{m-1}+n-1}}-X_{\theta_{m-1}}\le (n{-}1)(-a{+}\varepsilon){+}A,\nonumber\\&&\qquad
 X_{T_{j_{m-1}+n}}{-}X_{\theta_{m-1}}>x\bigr\}.
\end{eqnarray*}
Therefore,
\begin{eqnarray}\label{theta.m.upper}
\lefteqn{\P\{X_{\theta_m}-X_{\theta_{m-1}}>x,\ \theta_k<\infty,\ \theta_k\le\tau\}}\nonumber\\
&\le& p^{k-m}(A) \sum_{n=1}^\infty
\P\bigl\{T_{j_{m-1}+n}\le\theta_m,\ T_{j_{m-1}+n}\le\tau,\
Y_{j_{m-1}+n}>x{+}(n{-}1)(a{-}\varepsilon){-}A\bigr\}\nonumber\\
&=& p^{k-m}(A) \sum_{n=1}^\infty
\P\bigl\{T_{j_{m-1}+n}\le\theta_m,\ T_{j_{m-1}+n}\le\tau\bigr\}\nonumber\\&&\qquad
\overline F(x+(n-1)(a-\varepsilon)-A),
\end{eqnarray}
because the event
$\{T_{j_{m-1}+n}\le\theta_m,\ T_{j_{m-1}+n}\le\tau\} =
\{T_{j_{m-1}+n}\le\theta_m,\ \overline{T_{j_{m-1}+n}>\tau}\}$
does not depend on the future increments, see \eqref{tau.n.crp}.
Next, for $m\le k$ and $B<A$,
\begin{eqnarray}\label{T.T.P1.P2}
\lefteqn{\P\bigl\{T_{j_{m-1}+n}\le\theta_m,\ T_{j_{m-1}+n}\le\tau\bigr\}}\nonumber\\
&=& \P\bigl\{T_{j_{m-1}+n}\le\tau,\ T_{j_{m-1}+n}\le\theta_m,\nonumber\\&&\qquad
X_{T_{j_{m-1}+n}-0}-X_{T_{j_{m-1}}}\le A-B+(n-1)(-a+\varepsilon)\bigr\}\nonumber\\
&& +\P\bigl\{T_{j_{m-1}+n}\le\tau,\ T_{j_{m-1}+n}\le\theta_m,\nonumber\\&&\qquad
X_{T_{j_{m-1}+n}-0}-X_{T_{j_{m-1}}}>A-B+(n-1)(-a+\varepsilon)\bigr\}\nonumber\\
&=:& P_1+P_2.
\end{eqnarray}
Firstly,
\begin{eqnarray*}
P_1 &=& \P\bigl\{\overline{T_{j_{m-1}+n}>\tau},\ T_{j_{m-1}+n}\le\theta_m,\
X_{T_{j_{m-1}+n}-0}-X_{T_{j_{m-1}}}\le A-B+(n-1)(-a+\varepsilon)\bigr\}\\
&=& \frac{1}{1-p(B)}\P\bigl\{\overline{T_{j_{m-1}+n}>\tau},\
T_{j_{m-1}+n}\le\theta_m,\\
&&\hspace{15mm}  X_{T_{j_{m-1}+n}-0}-X_{T_{j_{m-1}}}\le A-B+(n-1)(-a+\varepsilon)\bigr\}\\
&&\hspace{15mm}
X_{T_{j_{m-1}+n+i-1}}-X_{T_{j_{m-1}+n}-0}\le B+i(-a+\varepsilon)\ \mbox{ for all }i\ge 1\bigr\}\\
&\le& \frac{1}{1-p(B)}\P\bigl\{\overline{T_{j_{m-1}+n}>\tau},\
X_{T_{j_{m-1}+n}}-X_{T_{j_{m-1}}}\le A+n(-a+\varepsilon)
\mbox{ for all }n\ge 1\bigr\}\\
&=& \frac{1}{1-p(B)}
\P\bigl\{T_{j_{m-1}+n}\le\tau,\ \theta_{m-1}<\infty,\ \theta_m=\infty\bigr\},
\end{eqnarray*}
where the second equality again follows from the independence of the future
increments \eqref{tau.n.crp} %, since $c=0$;
and from the definition of $p(B)$.
Secondly,
\begin{eqnarray*}
P_2 &\le& \P\bigl\{\theta_{m-1}<\infty,\
X_{T_{j_{m-1}+n}-0}-X_{T_{j_{m-1}}}\in (n-1)(-a+\varepsilon)+(A-B,A]\bigr\}\\
&=& \P\bigl\{\theta_{m-1}<\infty,\ S_{m,n-1}\in (A-B,A]\bigr\}\\
&=& p^{m-1}(A)\P\bigl\{S_{m,n-1}\in (A-B,A]\bigr\},
\end{eqnarray*}
where $S_{m,n}:=Y_{\theta_{m-1}+1}+\ldots+Y_{\theta_{m-1},n}+n(a-\varepsilon)$
is a random walk with negative drift $-\varepsilon$.
Hence, it follows from \eqref{theta.m.upper} and \eqref{T.T.P1.P2} that, for $m\le k$,
\begin{eqnarray*}
\lefteqn{\P\{X_{\theta_m}-X_{\theta_{m-1}}>x,\ \theta_k<\infty,\ \theta_k\le\tau\}}\\
&\le& \frac{p^{k-m}(A)}{1-p(B)} \sum_{n=1}^\infty
\P\{\tau\ge T_n,\ \theta_{m-1}<\infty,\ \theta_m=\infty\}\overline F(x+n(a-\varepsilon)-A-a)\\
&&+p^{k-1}(A)\sum_{n=1}^\infty \P\bigl\{S_{m,n-1}\in (A-B,A]\bigr\}
\overline F(x+n(a-\varepsilon)-A-a).
\end{eqnarray*}
By the Blackwell renewal theorem, for any fixed $B>0$,
\begin{eqnarray*}
\sum_{n=1}^\infty \P\bigl\{S_{1,n-1}\in (A-B,A]\bigr\}
&\to& 0\quad\mbox{as }A\to\infty.
\end{eqnarray*}
Thus, there exists a sufficiently large $A=A(B)$ such that
the last series is less than $\varepsilon$, hence by \eqref{p.A.to.0}
\begin{eqnarray}\label{upper.tau.crp.2.tails}
\lefteqn{\P\{X_{\theta_m}-X_{\theta_{m-1}}>x,\ \theta_k<\infty,\ \theta_k\le\tau\}}\nonumber\\
&\le& \frac{\varepsilon^{k-m}}{1-p(B)} \sum_{n=1}^\infty
\P\{\tau\ge T_n,\ \theta_{m-1}<\infty,\ \theta_m=\infty\}\overline F(x+n(a-\varepsilon)-A-a)\nonumber\\&&\qquad
+\varepsilon^k \overline F(x-A-a).
\end{eqnarray}

Next we have, for all Borel sets $B_1$, \ldots, $B_k\subseteq(A,\infty)$,
\begin{eqnarray*}
\lefteqn{\P\{X_{\theta_1}\in B_1,\ X_{\theta_2}-X_{\theta_1}\in B_2,
\ldots,X_{\theta_k}-X_{\theta_{k-1}}\in B_k,\ \theta_k<\infty,\ \theta_k\le\tau\}}\\
&\le& \P\{X_{\theta_1}\in B_1,\ \theta_1\le\tau,\ X_{\theta_2}-X_{\theta_1}\in B_2,
\ldots,X_{\theta_k}-X_{\theta_{k-1}}\in B_k,\ \theta_k<\infty\}\\
&=& \P\{X_{\theta_1\wedge\tau}\in B_1\} p^{k-1}
\prod_{m=2}^k \P\{X_{\theta_m}-X_{\theta_{m-1}}\in B_m\mid \theta_m<\infty\},
\end{eqnarray*}
again due to the independence of the future increments \eqref{tau.n.crp}.
Then, as follows from \eqref{eq:upper.1.crp},
\begin{eqnarray}\label{S.theta.1-}
&&\P\{X_{\theta_1}\in B_1,\ X_{\theta_2}{-}X_{\theta_1}\in B_2,
\ldots,X_{\theta_k}{-}X_{\theta_{k-1}}\in B_k,\ \theta_k<\infty,\ \theta_k\le\tau\}\nonumber\\
&& \qquad \le \overline G_\tau(B_1) p^{k-1} \prod_{m=2}^k G_\infty(B_m),
\end{eqnarray}
where the distribution
$G_\infty(B):=\P\{X_{\theta_1}\in B\mid \theta_1<\infty\}$ possesses the following
local upper bound, for some $c_\infty<\infty$,
\begin{eqnarray}\label{eq:G.infty.1}
G_\infty[y,y+1] &\le& c_\infty\overline F(y)\quad\mbox{for all }y>0,
\end{eqnarray}
owing to the condition $F\in\mathcal S^*$.

We can write down that
\begin{eqnarray*}
&&\P\{M_{\theta_k}>x,\ \theta_k<\infty,\ \theta_k\le\tau\}
\le \sum_{m=1}^k \P\{X_{\theta_m}-X_{\theta_{m-1}}>x,\ \theta_k<\infty,\ \theta_k\le\tau\}\\
&&\quad + \P\{M_{\theta_k}>x,\
X_{\theta_m}-X_{\theta_{m-1}}\le x\mbox{ for all }m\le k-1,\
\theta_k<\infty,\ \theta_k\le\tau\},
\end{eqnarray*}
which holds true even if the ladder heights are dependent given $\theta_k\le\tau$;
the latter is observed in most cases.
Since the events $\{\theta_k<\infty,\theta_{k+1}=\infty\}$, $k\ge 1$, make a partition
of the event $\{\theta_1<\infty\}$, by the total probability law,
we get the following upper bound
\begin{eqnarray}\label{upper.tau.s1.s2}
\P\{M_\tau>x\} &=& \sum_{k=1}^\infty
\P\{M_{\theta_k}>x,\ \theta_k<\infty,\ \theta_k\le\tau,\ \theta_{k+1}=\infty\}\nonumber\\
&\le& \sum_{k=1}^\infty \P\{M_{\theta_k}>x,\ \theta_k<\infty,\ \theta_k\le\tau\}\nonumber\\
&\le& \sum_{k=1}^\infty \sum_{m=1}^k
\P\{X_{\theta_m}-X_{\theta_{m-1}}>x,\ \theta_k<\infty,\ \theta_k\le\tau\}\nonumber\\
&&+ \sum_{k=2}^\infty\P\{M_{\theta_k}>x,\
X_{\theta_m}{-}X_{\theta_{m-1}}\le x\mbox{ for all }m\le k{-}1,\
\theta_k<\infty,\ \theta_k\le\tau\}\nonumber\\
&=:& \Sigma_1(x)+\Sigma_2(x).
\end{eqnarray}
It follows from \eqref{upper.tau.crp.2.tails} that
\begin{eqnarray}\label{eq:sigma1}
\Sigma_1(x)
&\le& \sum_{k=1}^\infty \sum_{m=1}^k \frac{\varepsilon^{k-m}}{1-p(B)}
\sum_{n=1}^\infty
\P\{\theta_{m-1}<\infty,\ \theta_m=\infty,\ \tau\ge T_n\}\overline F(x{-}A{-}a+n(a{-}\varepsilon))\nonumber\\&&\qquad
+\overline F(x-A-a)\sum_{k=1}^\infty \varepsilon^k \nonumber\\
&=& \frac{1}{(1{-}p(B))(1{-}\varepsilon)}
\sum_{n=1}^\infty \overline F(x{-}A{-}a{+}n(a{-}\varepsilon))
\sum_{m=1}^\infty
\P\{\theta_{m-1}{<}\infty,\ \theta_m{=}\infty,\ \tau\ge T_n\}\nonumber\\&&\qquad
+\overline F(x{-}A{-}a)\varepsilon/(1{-}\varepsilon) \nonumber\\
&=& \frac{1}{(1-p(B))(1-\varepsilon)}
\sum_{n=1}^\infty \overline F(x-A-a+n(a-\varepsilon))\P\{\tau\ge T_n\}
+\overline F(x)\frac{\varepsilon}{1-\varepsilon} \nonumber\\
&& = \quad \frac{1}{(1-p(B))(1-\varepsilon)}
\overline G_\tau(x)+\overline F(x-A-a)\frac{\varepsilon}{1-\varepsilon}
\quad\mbox{for all sufficiently large }x.
\end{eqnarray}
To bound the second term in \eqref{upper.tau.s1.s2}, we firstly notice that
\begin{eqnarray*}
M_{\theta_k} &\le& A+X_{\theta_1}^++(X_{\theta_2}-X_{\theta_1})^+
+\ldots+(X_{\theta_k}-X_{\theta_{k-1}})^+.
\end{eqnarray*}
Then let us apply \eqref{S.theta.1-} and \eqref{eq:G.tau.mu}:
\begin{eqnarray*}
\Sigma_2(x) &\le&
\sum_{k=2}^\infty \varepsilon^{k-1} \P\{M_{\theta_k}>x,\ \theta_1\le\tau,\
X_{\theta_m}{-}X_{\theta_{m-1}}\le x\mbox{ for all }m\le k\mid \theta_k<\infty\}\\
&\le& \sum_{k=2}^\infty \varepsilon^{k-1} \int_{D_k(x)}
G_\tau(dy_1)G_\infty(dy_2)\ldots G_\infty(dy_k)\\
&=& \sum_{k=2}^\infty \varepsilon^{k-1} \int_0^\infty \mu_\tau(dy)\int_{D_k(x)}
F(y+dy_1)G_\infty(dy_2)\ldots G_\infty(dy_k),
\end{eqnarray*}
where $D_k(x):=\{(y_1,\ldots,y_k)\in[0,x]^k:y_1+\ldots+y_k>x\}$.
Therefore, by the local bound \eqref{eq:G.infty.1} for the distribution $G_\infty$,
for some $c_1<\infty$,
\begin{eqnarray*}
\Sigma_2(x) &\le& c_1\sum_{k=2}^\infty (\varepsilon c_\infty)^{k-1}
\int_0^\infty \mu_\tau(dy)\int_{D_k(x)}
F(y+dy_1)\overline F(y_2)dy_2\ldots\overline F(y_k)dy_k\\
&=& c_1\E Y^+\sum_{k=2}^\infty (\varepsilon c_\infty \E Y^+)^{k-1}
\int_0^\infty \mu_\tau(dy)\int_{D_k(x)}
F^+(y{+}dy_1)\overline{F^+}(y_2)dy_2\ldots\overline{F^+}(y_k)dy_k,
\end{eqnarray*}
where $F^+(dz):=F(dz)/\E Y^+$ is a distribution on $\R^+$.
Integration against $y_1$ leads to an upper bound
\begin{eqnarray*}
\Sigma_2(x)
&\le& c_2 \sum_{k=2}^\infty (\varepsilon c_\infty \E Y^+)^{k-1}
\int_0^\infty \mu_\tau(dy)
\int_{[0,x]^{k-1}}
\overline{F^+}(x+y-y_2-\ldots-y_k)\nonumber\\&&\qquad\overline{F^+}(y_2)dy_2\ldots\overline{F^+}(y_k)dy_k.
\end{eqnarray*}
The function $\overline {F^+}(x)=\overline F(x)/\E Y^+$, $x>0$, is a subexponential density
function itself due to $F\in\mathcal S^*$, so we can apply Kesten's bound for
subexponential densities, see Foss et al.\ \cite[Theorem 4.11]{FKZ},
which guarantees that there exists an $C<\infty$ such that
\begin{eqnarray*}
\int_{[0,x]^{k-1}}
\overline{F^+}(x+y-y_2-\ldots-y_k)\overline{F^+}(y_2)dy_2\ldots\overline{F^+}(y_k)dy_k
&\le& C 2^k\overline{F^+}(x+y)
\end{eqnarray*}
for all $k\ge 2$ and $x$, $y>0$. Hence, for all $k\ge 2$,
\begin{eqnarray*}
\int_0^\infty \mu_\tau(y) \int_{D_k(x)}
F(y+dy_1)\overline F(y_2)dy_2\ldots\overline F(y_k)dy_k
&\le& C_2 2^k\overline G_\tau(x)
\end{eqnarray*}
for all sufficiently large $x$. In its turn it yields that
$\Sigma_2(x) \le \gamma\varepsilon\overline G_\tau(x)$, $\gamma<\infty$,
provided $2\varepsilon c_\infty<1$,
which being substituted together with \eqref{eq:sigma1} into \eqref{upper.tau.s1.s2}
implies the following upper bound
\begin{eqnarray*}
\P\{M_\tau>x\}
&\le&  \frac{1}{(1-p(B))(1-\varepsilon)}\overline G_\tau(x)
+\overline F(x-A-a)\frac{\varepsilon}{1-\varepsilon}+\gamma\varepsilon\overline G_\tau(x).
\end{eqnarray*}
Since we can choose $\varepsilon>0$ as small as we please
and $B$ as large as we please,
\begin{eqnarray*}
\P\{M_\tau>x\} &\le&  (1+o(1))\overline G_\tau(x)+o(\overline F(x))\quad\mbox{as }x\to\infty,
\end{eqnarray*}
uniformly for all $\tau\ge T_1$.
Since $\P\{\tau\ge T_1\}=1$, the last term can be incorporated into $\overline G_\tau(x)$,
which concludes the proof of Proposition \ref{l:stopping.+3.crp}
in the case $c\le 0$, owing to the asymptotics \eqref{eq:upper.4.crp} for $\overline G_\tau(x)$.

%The case $c<0$ can be reduced to the case $c=0$
%by considering a compound renewal process with
%jumps $Y_n'=Y_n+c(T_n-T_{n-1})$ and zero linear component.

Now proceed with the proof in the case $c>0$ under the condition \eqref{cond.o}.
We need to slightly modify the jumps,
because in the case $c>0$ it is not any longer true that the maximum can be only attained
at a jump epoch $T_n$, instead it can be attained just prior to that, at time epoch $T_n-0$.
However there exists a sufficiently large $z>0$ such that the truncated jumps $\widetilde Y_k:=\max(Y_k,-z)$
satisfy $\E\widetilde Y_k\le \E Y_k+\varepsilon$. Consider a modified process $\widetilde X$
with jumps $\widetilde Y_k$ instead of $Y_k$ dominates the original process $X$,
it is sufficient to prove a matching upper bound for a process with jumps
bounded below by, say, $-z$.

As in the case $c\le 0$, we again consider the random time
$\widetilde\tau:=\max\{T_n:\ T_n\le\tau\}$ which
does not depend on the future increments of $X$;
let $\nu$ be such that $\widetilde\tau=T_\nu$. For that reason, due to $c>0$,
\begin{eqnarray*}
M_\tau &\le& M_{\widetilde\tau}+c(\tau-\widetilde\tau)\ \le\
M_{\widetilde\tau}+c(T_{\nu+1}-T_\nu)\ =\
M_{\widetilde\tau}+c\eta
\end{eqnarray*}
where conditioning on $\widetilde\tau$ and Example \ref{example3.1.1} imply that $\eta$
is an independent random variable distributed as $T_1$
and the tail distribution of $c\eta$ is negligible compared to that of $F$,
by the condition \eqref{cond.o}. Notice that then
\begin{eqnarray}\label{remainder.little.o}
\lefteqn{\P\{c\eta>x\}+\int_0^x \P\{c\eta\in du\}\E\int_{x-u}^{x-u+aN_\tau}\overline F(y)dy}
\nonumber\\
&=& \P\{cT_1>x\}+\E\int_0^{aN_\tau} dy\int_0^x\P\{c\eta\in du\}\overline F(x+y-u)
\nonumber\\
&\le& \P\{cT_1>x\}+\E\int_0^{aN_\tau} dy\int_0^{x+y}\P\{c\eta\in du\}\overline F(x+y-u)
\nonumber\\
&=& \P\{cT_1>x\}+(1+o(1))\E\int_0^{aN_\tau} \overline F(x+y)dy\quad\mbox{as }x\to\infty,
\end{eqnarray}
uniformly for all $\tau$ because, due to \cite[Corollary 3.18]{FKZ},
\begin{eqnarray*}
\int_0^{x+y}\P\{c\eta\in du\}\overline F(x+y-u) &\sim&
\overline F(x+y)\quad\mbox{as }x\to\infty\mbox{ uniformly for all }y\ge 0.
\end{eqnarray*}
In addition, $N_\tau=N_{\widetilde\tau}$.
Therefore, without loss of generality it is enough to consider the case where
$\tau$ only takes values from $\{T_n,\ n\ge 0\}$.
Then, since the jumps are assumed to be bounded below by $-z$,
\begin{eqnarray*}
M_\tau &=& \max_{T_n\le\tau}\{M_{T_n},\ M_{T_n-0}\}\ \le\ \max_{T_n\le\tau}M_{T_n}+z.
\end{eqnarray*}
Therefore, by Lemma \ref{l:conditioning}, it suffices to consider the case where $\tau\ge T_1$.
This allows us to conclude the proof of Proposition \ref{l:stopping.+3.crp} in the case $c>0$
along the lines of the earlier proof  in the case $c\le 0$.
\end{proof}

In this section we have repeatedly made use of the following fact which is proved in Appendix.
Let $\P\{\tau>0\}>0$. Consider a new probability measure
$\widehat P\{B\}=\P\{B,\tau>0\}/\P\{\tau>0\}$.
Let $\widehat X$ be $X$ under the measure $\widehat P$,
and similar let $\widehat\tau$ be $\tau$ under the measure $\widehat P$.

\begin{Lemma}\label{l:conditioning}
Let $\tau$ be independent of the future increments of $X$ in the sense of Definition \ref{indinc}.
Then $\widehat X$ is stochastically equivalent to $X$,
$\widehat\tau$ does not depend on the future increments of $\widehat X$, and
\begin{eqnarray*}
\P\{M_\tau>x\} &=&
\P\{\tau>0\}\P\{\widehat M_{\widehat\tau}>x\}\quad\mbox{for all }x>0.
\end{eqnarray*}
Let $\tau$ only take values from $\{T_n\}$ and let $\tau$ be independent
of the future increments of $X$ in the sense of Definition \ref{defindfuture}.
Then $\widehat X$ is stochastically equivalent to $X$,
$\widehat\tau$ does not depend on the future increments of $\widehat X$
in the same sense, and
\begin{eqnarray*}
\P\{M_\tau>x\} &=&
\P\{\tau\ge T_1\}\P\{\widehat M_{\widehat\tau}>x\}\quad\mbox{for all }x>0.
\end{eqnarray*}
\end{Lemma}

\section{Application to compound Poisson process with linear component}
\label{taudep.cPp}

For a {\it compound Poisson process with linear component} $X$ where
$N=\{N_t, t\ge 0\}$ is a homogeneous Poisson process with intensity
of jumps $\lambda$, we have $\E N_t=t\lambda$
and $\P\{cT_1>x\}=e^{-\lambda x/c}=o(\overline F(x))$
provided $F$ is heavy-tailed. The following result holds.

\begin{Theorem}\label{thm:Poisson}
Let $X$ be a compound Poisson process with linear component.
If $\E(c/\lambda+Y_1) =: -a<0$ and the distribution $F$ of $Y_1^+$
is strong subexponential, then
\begin{eqnarray*}
\P \{M_\tau>x\} &=& \frac{1+o(1)}{a} \E\int_x^{x+aN_\tau}\overline F(v)dv
+O(e^{-\lambda x/c})\\
&=& \frac{1+o(1)}{a} \E \int_x^{x+a\lambda\tau}\overline F(v)dv+O(e^{-\lambda x/c})\\
&=& \frac{1+o(1)}{|\E X_1|} {\E}
 \int_x^{x+|\E X_1|\tau}\P\{X_1>v\}dv+O(e^{-\lambda x/c})\quad\mbox{as }x\to\infty,
\end{eqnarray*}
uniformly for all random times $\tau\in[0,\infty]$
that do not depend on the future increments of $X$.
\end{Theorem}

\begin{proof}
The first equivalence in the theorem follows directly from  the discussion on the condition \eqref{tau.n.crp} in Introduction and from Theorem \ref{th:stopping.1.crp}.
We prove the second equivalence now, doing so in several steps.

We start from the following analogue of the Wald--Kolmogorov--Prokhorov identity
(see \cite{KP}) that holds for L\'evy processes, which has independent own interest.

\begin{Lemma}\label{l:E.tau.f}
Let $X$ be a L\'evy process with finite drift $m:=\E X_1$, and let $\tau$
be a random time that does not depend on the future increments of $X$.
If $\E\tau<\infty$ then $\E X_\tau = m\E\tau$.

In addition, if $\sigma^2={\mathbb Var} X_1<\infty$, then
$\E(X_\tau-m\tau)^2 = \sigma^2\E\tau$.
\end{Lemma}

%Its proof is given in Appendix.
Based on Lemma \ref{l:E.tau.f}, one can prove the following asymptotic equivalence.
%- its proof is also demonstrated in Appendix.

\begin{Lemma}\label{l:stopping.T}
Let $T>0$ and $\mathcal T_T=\{\tau:\tau\le T\}$ be a family of random
times that do not depend on the future increments of $N$. If $F$ is long-tailed then
\begin{eqnarray*}
\E\int_0^{aN_\tau} \overline F(x+y)dy\ \sim\
\E\int_0^{a\lambda\tau} \overline F(x+y)dy &\sim& a\lambda\E\tau \overline F(x)
\end{eqnarray*}
as $x\to\infty$ uniformly for all $\tau\in\mathcal T_T$.
\end{Lemma}
The proofs of Lemmas \ref{l:E.tau.f} and
\ref{l:stopping.T} are given in Appendix.

Then the following result is straightforward.

\begin{Lemma}\label{l:E.length.U+.eps}
Let $X$ be a L\'evy process with finite drift $m:=\E X_1$ and diffusion coefficient
$\sigma^2={\mathbb Var} X_1$. Then, for any fixed $\varepsilon>0$,
$\sup_\tau \E(X_\tau-(m+\varepsilon)\tau) <\infty$.
\end{Lemma}

\begin{proof}
The L\'evy process $Y_t=X_t-(m+\varepsilon)t$, $t\ge 0$ is negatively driven and its jumps are
square integrable, so
$\E \sup_{t>0} Y_t <\infty$, hence the result follows.
\end{proof}

We are ready to prove  the second equivalence in the statement of Theorem \ref{thm:Poisson}.
For a general $\tau$, fix an $\varepsilon>0$ and a $T<\infty$, and consider
the following decomposition and the upper bound
\begin{eqnarray}\label{E.Pois.deco}
\E\int_x^{x+aN_\tau}\overline F(v)dv &=&
\E\int_x^{x+aN_{\tau\wedge T}}\overline F(v)dv
+\E\Bigl\{\int_{x+aN_T}^{x+aN_\tau}\overline F(v)dv;\ \tau>T\Bigr\}\nonumber\\
&\le& \E\int_x^{x+aN_{\tau\wedge T}}\overline F(v)dv
+\E\Bigl\{\int_x^{x+a(N_\tau-N_T)}\overline F(v)dv;\ \tau>T\Bigr\}
\nonumber\\
&=& \E\int_x^{x+aN_{\tau\wedge T}}\overline F(v)dv
+\E \int_x^{x+a(N_\tau-N_T)^+}\overline F(v)dv
\ =:\ E_1+E_2.
\end{eqnarray}
Fix an $\varepsilon>0$. Then
\begin{eqnarray*}
E_2 &\le&
\E\int_x^{x+a(\lambda+\varepsilon)(\tau-T)^+}\overline F(v)dv
+\overline F(x)a\E[(N_\tau-N_T)^+-(\lambda+\varepsilon)(\tau-T)^+]^+\\
&\le&
\E\int_x^{x+a(\lambda+\varepsilon)(\tau-T)^+}\overline F(v)dv
+\overline F(x)a\E\Bigl\{\sup_{s\ge 0}\{(N_{T+s})-N_T)-(\lambda+\varepsilon)s\}; \tau>T\Bigr\}.
\end{eqnarray*}
We recall that the process $\widehat N_s=N_{T+s}-N_T$, $s\ge 0$,
is independent of the event $\tau>T$. Hence by Lemma \ref{l:E.length.U+.eps},
\begin{eqnarray*}
E_2 &\le&
\E\int_x^{x+a(\lambda+\varepsilon)(\tau-T)^+}\overline F(v)dv
+c(\varepsilon)\P\{\tau>T\}\overline F(x),
\end{eqnarray*}
where the constant $c(\varepsilon)$ depends on neither $\tau$ nor $T$. Further,
\begin{eqnarray*}
\E\int_x^{x+a(\lambda+\varepsilon)(\tau-T)^+}\overline F(v)dv
&\le& \frac{\lambda+\varepsilon}{\lambda}\E\int_x^{x+a\lambda(\tau-T)^+}\overline F(v)dv,
\end{eqnarray*}
because the tail function $\overline F(x)$ is decreasing. Hence,
\begin{eqnarray*}
E_2 &\le& (1+\varepsilon/\lambda)\E\Bigl\{\int_x^{x+a\lambda(\tau-T)}\overline F(v)dv;\
\tau>T\Bigr\}
+\P\{\tau>T\}c(\varepsilon)\overline F(x)\\
&\le& (1+\varepsilon/\lambda)\E\Bigl\{\int_x^{x+a\lambda\tau}\overline F(v)dv;\ \tau>T\Bigr\}
+\P\{\tau>T\}c(\varepsilon)\overline F(x).
\end{eqnarray*}
Therefore, due to the long-tailedness of $F$,
there exists an $\widehat x=\widehat x(T)$ such that, for $x\ge\widehat x$,
\begin{eqnarray*}
E_2 &\le& (1+\varepsilon/\lambda)\E\Bigl\{\int_x^{x+a\lambda\tau}
\overline F(v)dv;\ \tau>T\Bigr\}
+\P\{\tau>T\}2c(\varepsilon)\overline F(x+a\lambda T).
\end{eqnarray*}
By Lemma \ref{l:stopping.T}, for all $\tau$ and for all sufficiently large $x$,
\begin{eqnarray*}
E_1 &\le& (1+\varepsilon)\E\int_x^{x+a\lambda(\tau\wedge T)} \overline F(y)dy\\
&=& (1+\varepsilon)\E\Bigl\{\int_x^{x+a\lambda\tau} \overline F(y)dy;\ \tau\le T\Bigr\}
+(1+\varepsilon)\P\{\tau>T\}\overline F(x+a\lambda T).
\end{eqnarray*}
Thus, for all $\tau$ and sufficiently large $x$,
\begin{eqnarray*}
E_1+E_2 &\le& (1+\widehat\varepsilon)\E\int_x^{x+a\lambda\tau}\overline F(v)dv
+\P\{\tau>T\}\widehat c\overline F(x+a\lambda T),
\end{eqnarray*}
where $\widehat\varepsilon=\varepsilon\max(1,1/\lambda)$ and
$\widehat c=1+\varepsilon +2c(\varepsilon)$. Since
\begin{eqnarray*}
\P\{\tau>T\}\overline F(x+a\lambda T) &\le& \frac{1}{a\lambda T}
\E\Bigl\{\int_x^{x+a\lambda\tau}\overline F(v)dv;\ \tau>T\Bigr\},
\end{eqnarray*}
we conclude an upper bound
\begin{eqnarray*}
E_1+E_2 &\le& (1+\widehat\varepsilon+\widehat c/a\lambda T)
\E\int_x^{x+a\lambda\tau}\overline F(v)dv.
\end{eqnarray*}
Firstly letting $T\to\infty$ and then $\varepsilon\downarrow 0$,  we derive from
\eqref{E.Pois.deco} that
\begin{eqnarray*}
\E\int_x^{x+aN_\tau}\overline F(v)dv &\le&
(1+o(1)) \E\int_x^{x+a\lambda\tau}\overline F(v)dv
\end{eqnarray*}
as $x\to\infty$ uniformly for all $\tau$ that do not depend on the future
increments of $N_t$.

To get a matching lower bound we start with the inequality
\begin{eqnarray*}
\E\int_x^{x+aN_\tau}\overline F(v)dv &\ge&
\E\int_x^{x+aN_{\tau\wedge T}}\overline F(v)dv
+\E\Bigl\{\int_{x+aN_T}^{x+aN_\tau}\overline F(v)dv;\
\tau>T,N_T\le(\lambda-\varepsilon)T\Bigr\};
\end{eqnarray*}
further arguments are quite similar to that used for the analysis
of the right hand side terms in the upper bound \eqref{E.Pois.deco}.
\end{proof}

\section{Proof of Theorem  \ref{th:stopping.1.lp} for L\'evy process}
\label{taudep.lp}

Given the distribution of $X_1$ is infinitely divisible,
recall the L\'evy--Khintchine formula for its characteristic exponent
$\Psi(\theta):=\log\E e^{i\theta X_1}$, for all $\theta\in\R$,
\begin{eqnarray*}
\Psi(\theta) &=&
%=
\Bigl(i\alpha\theta-\frac{1}{2}\sigma^2\theta^2\Bigr)
+\int_{0<|x|<1}(e^{i\theta x}{-}1{-}i\theta x)\Pi(dx)
+\int_{|x|\ge 1}(e^{i\theta x}{-}1)\Pi(dx)\\
&=:& \Psi_1(\theta)+\Psi_2(\theta)+\Psi_3(\theta);
\end{eqnarray*}
see, e.g. Kyprianou \cite[Sect. 2.1]{Kyprianou}.
Here $\Pi$ is the L\'evy measure concentrated on $\R\setminus\{0\}$
and satisfying $\int_\R(1\wedge x^2)\Pi(dx)<\infty$.
Let $X^{(1)}$, $X^{(2)}$ and $X^{(3)}$ be independent
processes given in the L\'evy--It\^o decomposition
$X_t\stackrel{d}{=} X^{(1)}_t+X^{(2)}_t+X^{(3)}_t$, $t\ge 0$, where $X^{(1)}$ is
a zero-drifted Brownian motion with characteristic exponent given by $\Psi^{(1)}$,
$X^{(2)}$ a square integrable martingale
with an almost surely countable number of jumps on each finite time
interval which are of magnitude less than unity and with
characteristic exponent given by $\Psi^{(2)}$,
and $X^{(3)}$ a compound Poisson process with linear component $-mt$, with intensity
$\lambda:=\Pi(\R\setminus(-1,1))$,
and with jump distribution $F(dx)=\Pi(dx)/\lambda$ concentrated on $\R\setminus(-1,1)$.
It is known---see, e.g. Kyprianou \cite[Theorem 3.6]{Kyprianou}
or Sato \cite[Theorem 25.17]{Sato}---that the process
$Z_t:=X^{(1)}_t+X^{(2)}_t$, $t\ge 0$ possesses all exponential moments
finite which allows us to show the following result,
see e.g. Corollary 8 in Korshunov \cite{K2018}.

\begin{Proposition}\label{cor:X.via.Pi}
{\rm(i)} The distribution of $X_1$ is long-tailed if and only if
the distribution of $X^{(3)}_1$ is so. In both cases,
$\P\{X_1>x\}\sim\P\{X^{(3)}_1>x\}$ as $x\to\infty$.

{\rm(ii)} The distribution of $X_1$ is strong subexponential
if and only if the distribution $F$ is so. In both cases,
$\P\{X_1>x\}\sim\Pi(x,\infty)$ as $x\to\infty$.
\end{Proposition}

\begin{proof}[Proof of Theorem \ref{th:stopping.1.lp}]
Since $X_1$ is assumed to be strong subexponential,
by Proposition \ref{cor:X.via.Pi} the distribution $F$ is strong subexponential too
and $\overline\Pi(x)\sim\P\{X_1>x\}\sim\lambda\overline F(x)$ as $x\to\infty$.

The proof is split into two parts, where we obtain matching lower and upper bounds.
We start with the lower bound.

Consider the sequence of jump epochs $T_n$ of the compound
Poisson component $X^{(3)}$, that is when $Y_n:=X_{T_n}-X_{T_n-0}\in\R\setminus(-1,1)$.
Then, as discussed in Introduction, the $\tau$ satisfies the condition \eqref{tau.n.crp}.
The distribution $F$ of $Y$ is long-tailed by Proposition \ref{cor:X.via.Pi}.

For the lower bound for $\P\{M_\tau>x\}$, let us consider a compound Poisson process
$\widehat X_t$ with jumps $X_{T_n}-X_{T_{n-1}}$ at time epochs $T_n$,
which are distributed as $X_{T_1}$. By the construction, $M_\tau\ge \widehat M_\tau$.
Application of Theorem \ref{thm:Poisson} to the compound Poisson process $\widehat X$
yields the required lower tail bound for $M_\tau$,
\begin{eqnarray}\label{lower.Levy}
\P\{M_\tau>x\}
&\ge& \P\{\widehat M_\tau>x\}
\ \sim\ \frac{1+o(1)}{m} \E\int_x^{x+m\tau}\P\{X_1>y\}dy\quad\mbox{as }x\to\infty,
\end{eqnarray}
uniformly for all $\tau$ independent of the future increments of $X$.

Now let us proceed with a matching upper bound.
Firstly, consider the random time $\widetilde\tau:=\max\{T_n:\ T_n\le\tau\}$ which
does not depend on the future increments of $X$ in the sense
of Definition \ref{defindfuture}; let $\nu$ be such that $\widetilde\tau=T_\nu$. For that reason,
\begin{eqnarray*}
M_\tau &\le& M_{\widetilde\tau}+\max_{t\in[T_\nu,T_{\nu+1})}(X_t-X_\nu)\ =\
M_{\widetilde\tau}+\eta
\end{eqnarray*}
where conditioning on $\widetilde\tau$ and Example \ref{example3.1.1} imply that
$\eta$ is an independent random variable distributed as $M_{T_1-0}$.
Since only $X^{(1)}$ and $X^{(2)}$ contribute to the value of $M_{T_1-0}$ and since
the exponentially distributed time interval $[0,T_1)$ is independent of them,
the distribution of the random variable $M_{T_1-0}$ is light-tailed;
hence the tail distribution of $c\eta$ is negligible compared to that of $F$,
so we can perform an upper bound \eqref{remainder.little.o}.
In addition, $N_\tau=N_{\widetilde\tau}$. Therefore, it suffices to show
a matching upper bound for $M_{\widetilde\tau}$, so from now on we assume
that $\tau$ only takes values from $\{T_n,\ n\ge 0\}$.

If $\tau$ is determined by the compound Poisson process $X^{(3)}$
with linear component $-mt$, then $\tau$ is independent of
both $X^{(1)}$ and $X^{(2)}$. Then we have the following upper bound
\begin{eqnarray*}
M_\tau &\le& \sup_{t\le\tau}(X^{(3)}_t+\varepsilon t)
+\sup_{t>0}(X^{(1)}_t+X^{(2)}_t-\varepsilon t),
\end{eqnarray*}
where the second supremum is independent of the first one.
Since both $X^{(1)}$ and $X^{(2)}$ are light-tailed,
the distribution of the second overall supremum is light-tailed too.
For the first supremum we apply Theorem \ref{thm:Poisson}
for a compound Poisson process with linear component,
and hence conclude the desired upper bound
which together with the lower bound \eqref{lower.Levy}
implies the required asymptotics.
\end{proof}

\section*{Acknowledgements}
The work is partially supported by the Polish National Science Centre
under the grant 2021/41/B/HS4/00599.

The authors are very thankful to the two referees whose comments
helped a lot to improve the paper.

\section*{Appendix}
%\subsection{Proof of Lemma \ref{l:E.tau.f}}

\begin{proof}[Proof of Lemma \ref{l:conditioning}]
We consider the case of Definition \ref{indinc}.
Since the event $\{\tau>0\}=\overline{\{\tau=0\}}$ does not depend on the $\sigma$-algebra
generated by the increments $X_t-X_0=X_t$, the conditioning on this event does not
change the distribution of the process. Next, let $t>0$, let $A\in\sigma\{X_s,s\le t\}$,
and let $B\in\sigma\{X_{t+s}-X_t,s>0\}$. Then, by the definition of $\widehat P$,
\begin{eqnarray*}
\widehat\P\{A,\widehat\tau>t,B\} &=&
\frac{\P\{A,\tau>t,B,\tau>0\}}{\P\{\tau>0\}}
\ =\ \frac{\P\{A,\tau>t,B\}}{\P\{\tau>0\}}
\ =\ \frac{\P\{A,\tau>t\}\P\{B\}}{\P\{\tau>0\}},
\end{eqnarray*}
owing to the independence of the future increments of $\tau$.
Again by the definition of $\widehat P$,
\begin{eqnarray*}
\frac{\P\{A,\tau>t\}\P\{B\}}{\P\{\tau>0\}} &=&
\widehat\P\{A,\widehat\tau>t\}\P\{B\}\ =\ \widehat\P\{A,\widehat\tau>t\}\widehat\P\{B\},
\end{eqnarray*}
since $\widehat X$ is stochastically equivalent to $X$.
\end{proof}

\begin{proof}[Proof of Lemma \ref{l:E.tau.f}]
For any fixed $T$, the random variable $\tau\wedge T$ also does not depend
on the future increments of $X$ as $\tau$, hence
\begin{eqnarray*}
\E (X_T-X_{\tau\wedge T}) &=& \E\E \{X_T-X_{\tau\wedge T}\mid\tau\}
\ =\ m\E(T-\tau\wedge T).
\end{eqnarray*}
Therefore,
\begin{eqnarray*}
mT &=& \E X_T
\ =\ \E X_{\tau\wedge T}+\E (X_T-X_{\tau\wedge T})
\ =\  \E X_{\tau\wedge T}+mT-m\E(\tau\wedge T),
\end{eqnarray*}
which implies $\E X_{\tau\wedge T}=m\E(\tau\wedge T)$ for all $T$.
Hence, by the Lebesgue monotone convergence theorem applied to $\tau\wedge T$,
$\E X_{\tau\wedge T} \to m\E\tau$ as $T\to\infty$.
On the other hand, $\E X_{\tau\wedge T}\to\E X_\tau$
because the random variable $X_\tau$ is integrable. Indeed,
\begin{eqnarray*}
|X_\tau| &\le& \sum_{n=1}^{[\tau]+1} Y_n,
\end{eqnarray*}
where the random variables
\begin{eqnarray*}
Y_n &:=& \sup_{s\in(0,1]} |X_{n-s}-X_{n-1}|,\quad n\ge 1,
\end{eqnarray*}
are i.i.d.\ with finite mean value because $X$ is a L\'evy process with finite drift.
Since $[\tau]+1$ does not depend on the future of
the sequence $\{Y_n\}$, by the Kolmogorov--Prokhorov equality,
\begin{eqnarray*}
\E \sum_{n=1}^{[\tau]+1} Y_n &=& (\E[\tau]+1)\E Y_1\ <\ \infty,
\end{eqnarray*}
so the proof of the first statement is complete.

For the second statement, firstly notice that, for a bounded $\tau$,
that is $\tau\le T$ for some $T<\infty$,
\begin{eqnarray*}
\sigma^2T &=& \E(X_T-mT)^2\\
&=& \E(X_\tau-m\tau)^2+2\E(X_\tau-m\tau)(X_T-X_\tau-m(T-\tau))
+\E(X_T-X_\tau-m(T-\tau))^2.
\end{eqnarray*}
Conditioning on $\tau$ implies that
\begin{eqnarray*}
\E(X_\tau-m\tau)(X_T-X_\tau-m(T-\tau))
&=& \E\E\{(X_\tau-m\tau)(X_T-X_\tau-m(T-\tau)\mid\tau\}
\ =\ 0,
\end{eqnarray*}
because $X_T-X_\tau-m(T-\tau)$ is independent of $X_\tau-m\tau$ given $\tau$.
Similarly,
\begin{eqnarray*}
\E(X_T-X_\tau-m(T-\tau))^2 &=& \E\E\{(X_T-X_\tau-m(T-\tau))^2\mid\tau\}
\ =\ \sigma^2\E(T-\tau).
\end{eqnarray*}
Combining the last three equalities we conclude that
\begin{eqnarray*}
\E(X_\tau-m\tau)^2 &=& \sigma^2T-\E(X_T-X_\tau-m(T-\tau))^2
\ =\ \sigma^2\E\tau
\end{eqnarray*}
for any bounded $\tau$. For an unbounded $\tau$,
we apply it to $\tau\wedge T$ and then let $T\to\infty$
as in the proof of the first statement.
\end{proof}

%\subsection{Proof of Lemma \ref{l:stopping.T}}
\begin{proof}[Proof of Lemma \ref{l:stopping.T}]
As we know from Lemma \ref{l:E.tau.f}, $\E N_\tau=\lambda\E\tau$
and $\E (N_\tau-\lambda\tau)^2=\lambda\E\tau$. Therefore,
\begin{eqnarray*}
\E N_\tau^2 &=& \E (N_\tau-\lambda\tau)^2+2\lambda\E N_\tau \tau-\lambda^2\E\tau^2
\ \le\ \E (N_\tau-\lambda\tau)^2+2\lambda T\E N_\tau.
\end{eqnarray*}
due to $\tau\le T$. Thus,
\begin{eqnarray*}
\E N_\tau^2 &\le& \lambda\E\tau+2\lambda T\lambda\E\tau
\ =\ (2T\lambda^2+\lambda)\E\tau,
\end{eqnarray*}
which in turn yields, for all $A>0$,
\begin{eqnarray*}
\E\{N_\tau;\ N_\tau>A\} &\le& \E N_\tau^2/A
\ \le\ (2T\lambda^2+\lambda)\E\tau/A,
\end{eqnarray*}
so, owing to Lemma \ref{l:E.tau.f} again,
\begin{eqnarray}\label{N.tau.le.A}
\E\{N_\tau;\ N_\tau\le A\} &=& \E N_\tau-\E\{N_\tau;\ N_\tau>A\}
\ \ge\ (\lambda-(2T\lambda^2+\lambda)/A)\E\tau.
\end{eqnarray}
Further, on the one hand,
\begin{eqnarray}\label{int.upper.2}
\E\int_x^{x+aN_\tau}\overline F(v)dv &\le& \overline F(x)\E(aN_\tau)
\ =\ a\lambda\E\tau\overline F(x).
\end{eqnarray}
On the other hand, for any fixed $A$,
\begin{eqnarray*}
\E\int_x^{x+aN_\tau}\overline F(v)dv
&\ge& \E\Bigl(\int_x^{x+aN_\tau}\overline F(v)dv;\ N_\tau\le A\Bigr)\\
&\ge& \overline F(x+aA)\E\{aN_\tau;\ N_\tau\le A\}
\ \ge\ (a\lambda-(2Ta\lambda^2+\lambda)/A)\E\tau\overline F(x+aA),
\end{eqnarray*}
by \eqref{N.tau.le.A}. Thus,
for any fixed $\varepsilon>0$, we can choose a sufficiently large A such that
\begin{eqnarray}\label{int.lower.2}
\E\int_x^{x+aN_\tau}\overline F(v)dv
&\ge& (a\lambda-\varepsilon/2)\E\tau\overline F(x+aA)
\ \ge\ (a\lambda-\varepsilon)\E\tau\overline F(x)
\end{eqnarray}
for all sufficiently large $x$, due to the long-tailedness of the distribution $F$.
Combining the bounds \eqref{int.upper.2} and \eqref{int.lower.2},
we conclude the first uniform asymptotics stated in the lemma,
\begin{eqnarray*}
\E\int_x^{x+aN_\tau}\overline F(v)dv &\sim& a\lambda\E\tau\overline F(x)
\quad\mbox{as }x\to\infty.
\end{eqnarray*}

Taking into account that
\begin{eqnarray*}
a\lambda\E\tau\overline F(x+T) &\le&
\E\int_x^{x+a\lambda\tau}\overline F(v)dv
\ \le\ a\lambda\E\tau\overline F(x),
\end{eqnarray*}
we also conclude the second uniform equivalence of the lemma,
\begin{eqnarray*}
\E\int_x^{x+a\lambda\tau}\overline F(v)dv &\sim& a\lambda\E\tau\overline F(x)
\quad\mbox{as }x\to\infty.
\end{eqnarray*}
\end{proof}

\end{document}